\documentclass[12pt]{amsart}
\usepackage{amssymb, amsmath, amsfonts, microtype, cite, tikz, graphicx, tabu, yhmath}
\usepackage[colorlinks]{hyperref}
\usepackage[a4paper]{geometry}
\usepackage{xpatch}
\xpatchcmd{\itemize}{\def\makelabel}{\setlength{\itemsep}{5pt}\def\makelabel}{}{}
\xpatchcmd{\enumerate}{\def\makelabel}{\setlength{\itemsep}{5pt}\def\makelabel}{}{}
\usepackage{enumitem}
\setlist[enumerate,itemize]{topsep=4pt, parsep=3pt, itemsep=3pt, 
leftmargin=0.5cm, 
listparindent=20pt, labelsep=5pt, itemindent=0pt, partopsep=5pt}

\usepackage{mathtools}
\DeclarePairedDelimiter\abs{\lvert}{\rvert}
\def\bg#1{\bigl({#1}\bigr)}

\def\bgg#1{\biggl({#1}\biggr)}
\def\C{\mathbb{C}}

\def\R{\mathbb{R}}

\def\rmand{\quad\hbox{ and }\quad}
\DeclarePairedDelimiter\babs{{\bigl\lvert}}{{\bigr\rvert}}
\def\fl#1{\lfloor{#1}\rfloor}
\def\bfl#1{\bigl\lfloor\,{#1}\,\bigr\rfloor}

\def\cl#1{\lceil{#1}\rceil}

\numberwithin{equation}{section}
\numberwithin{figure}{section}

\def\AAM{Adv.\ in Appl.\ Math.\ }

\def\AMS{Amer.\ Math.\ Soc.}
\def\ANM{Appl.\ Numer.\ Math.\ }

\def\JMAA{J.\ Math.\ Anal.\ Appl.\ }

\DeclareMathOperator\sgn{\mathrm{sgn}}

\usepackage[capitalize]{cleveref}
\crefname{case}{Case}{Cases}
\crefname{case}{Case}{Cases}
\creflabelformat{case}{~\upshape#2#1#3}
\crefname{cond}{Condition}{Conditions}
\Crefname{cond}{Condition}{Conditions}
\creflabelformat{cond}{~\upshape(#2#1#3)}
\crefname{ineq}{Ineq.}{Ineqs.}
\Crefname{ineq}{Inequality}{Inequalities}
\creflabelformat{ineq}{~\upshape(#2#1#3)}
\crefname{rec}{Recurrence}{Recurrences}
\Crefname{rec}{Recurrence}{Recurrences}
\creflabelformat{rec}{~\upshape(#2#1#3)}
\crefname{icond}{Condition}{Conditions}
\Crefname{icond}{Condition}{Conditions}
\creflabelformat{icond}{~\upshape#2#1#3}
\crefname{thm}{Theorem}{Theorems}
\Crefname{thm}{Theorem}{Theorems}
\creflabelformat{thm}{~\upshape#2#1#3}

\crefrangeformat{equation}{Eqs.\!~(#3#1#4) ---~(#5#2#6)}
\crefrangeformat{ineq}{Ineqs.\!~(#3#1#4) ---~(#5#2#6)}

\newtheorem{thm}{Theorem}[section]

\newtheorem{lem}[thm]{Lemma}
\newtheorem{dfn}[thm]{Definition}
\newtheorem{eg}[thm]{Example}
\newtheorem{conj}[thm]{Conjecture}

\theoremstyle{definition}
\newtheorem{case}{Case}[thm]

\makeatletter \@addtoreset{equation}{section} \makeatother

\allowbreak
\allowdisplaybreaks

\newcommand{\noo}{normalised polynomial sequence satisfying \cref{rec11}}

\begin{document}

\title[]{{Piecewise interlacing zeros of polynomials}}

\author[D.G.L. Wang]{David G.L. Wang$^\dag$$^\ddag$}
\address{
$^\dag$School of Mathematics and Statistics, Beijing Institute of Technology, 102488 Beijing, P. R. China\\
$^\ddag$Beijing Key Laboratory on MCAACI, Beijing Institute of Technology, 102488 Beijing, P. R. China}
\email{david.combin@gmail.com}

\author[J.J.R. Zhang]{Jerry J.R. Zhang}
\address{
$^\dag$School of Mathematics and Statistics, Beijing Institute of Technology, 102488 Beijing, P. R. China}
\email{jrzhang.combin@gamil.com}
\thanks{\hskip-11pt 
D.G.L. Wang is supported by the Beijing Institute of Technology Research Fund Program for Young Scholars (grant No. 2015CX04016), the General Program of National Natural Science Foundation of China (grant No. 11671037), and the Fundamental Research Funds for the Central Universities (grant No. 20161742027).}


\begin{abstract}   
We introduce the concept of piecewise interlacing zeros for studying the relation of root distribution of two polynomials. The concept is pregnant with an idea of confirming the real-rootedness of polynomials in a sequence. Roughly speaking, one constructs a collection of disjoint intervals such that one may show by induction that consecutive polynomials have interlacing zeros over each of the intervals. We confirm the real-rootedness of some polynomials satisfying a recurrence with linear polynomial coefficients. This extends Gross et al.'s work where one of the polynomial coefficients is a constant.
\end{abstract}

\maketitle

\tableofcontents
\parskip 9pt

\section{Introduction}

The root distribution of a single polynomial is a long-standing topic all along the history of mathematics; 
see Rahman and Schmeisser's book \cite{RS02B}. 
For instance, the significance of real-rootedness and stability of polynomials 
can be found from Stanley~\cite[\S 4]{Sta00} and Borcea and Br\"anden~\cite{BB09}.

Motived by the LCGD conjecture from topological graph theory,
Gross, Mansour, Tucker and the first author~\cite{GMTW16-01,GMTW16-10}
studied the root distribution of polynomials satisfying some recurrences 
of order two, with one of the polynomial coefficients in the recurrence linear and the other constant. 
In this paper, we continue to study the root distribution of polynomials defined by a recurrence of order two,
but with both the polynomial coefficients linear. 
It turns out that this change brings much richer root geometry.
Orthogonal polynomials have a closed relation with
such recurrences; see Szeg\H o~\cite{Sze59B}, 
Andrews, Richard and Ranjan~\cite{ARR99B},
and Stahl and Totik~\cite{ST92B}. 
Equally closed connection with quasi-orthogonal polynomials
was pointed by Brezinski, Driver and Redivo-Zaglia~\cite{BDR04}.

A popular way to show the real-rootedness of all polynomials in a sequence 
is to show that consecutive polynomials have interlacing real zeros by induction; see Liu and Wang~\cite{LW07}. 
The classical notion of interlacing zeros can be found from 
Rahman and Schmeisser's book~\cite{RS02B}.
Unfortunately and commonly encountered,
consecutive polynomials in a sequence, such as those in \cref{sec:main},
do not have interlacing zeros over the whole real line.

We introduce the concept of piecewise interlacing zeros for studying the relation of root distribution of two polynomials. The concept is pregnant with an idea of confirming the real-rootedness of polynomials in a sequence.
Roughly speaking, one constructs a collection of disjoint intervals such that one may show by induction that consecutive polynomials have interlacing zeros over each of the intervals. 

We should mention that He and Saff~\cite{HS94} 
considered the root distributions of sequences of Faber polynomials 
associated with a compact set on the complex plane. 
By virtue of choosing the compact set to be the $m$-cusped hypocycloid with the parametric equation
\[
z=e^{i\theta}+\frac{1}{m-1}e^{(1-m)i\theta},
\] 
the roots of corresponding Faber polynomials are located on the $m$ line segments
each of which connects the origin and a complex number $\frac{m}{m-1}e^{2k\pi i/m}$ for some $k=0,1,\ldots,m-1$.
They proved that consecutive Faber polynomials for any given~$m$ interlace on each line segment.

The idea of piecewise interlacing zeros presented in this paper has an essential 
difference from the one of He and Saff's. 
Our idea needs a discovery of disjoint intervals from a single line to form the pieces,
such that the real-rootedness of each polynomial on each resulting interval can be shown by induction.
In comparison, He and Saff's way of forming the pieces is somehow straightforward, 
for at least zeros lying on distinct lines are not said to be interlacing,
although it is literally true to say that the zeros are ``piecewise'' interlacing for the above Faber polynomials.

Furthermore, the extensive concept of piecewise interlacing 
inherits some attributes of the classical notion of interlacing property.  
For example, Jordaan and To\'okos \cite{JT09} call the interlacing 
properties among any two of the zero sets of three consecutive polynomials in the sequence
the ``triple interlacing property''. In \cref{thm:pl}, the piecewise interlacing properties
is shown to have such triple behaviour either. 

This paper is organised as follows. 
\Cref{sec:piecewise} explains the idea of confirming the real-rootedness
of polynomials in a sequence based on the concept of piecewise interlacing zeros.
In \S\,\ref{sec:main}, 
we give an application of the idea as \cref{thm:rr} with some numerical examples.
A proof of \cref{thm:rr} lies in \S\,\ref{sec:pf}. 
We end this paper with a conjecture in \S\,\ref{sec:conj}.

\section{Piecewise interlacing zeros}\label{sec:piecewise}

The notion of interlacing zeros can be found from Rahman and Schmeisser's book~\cite[Definition 6.3.1]{RS02B}. 

\begin{dfn}
We say that an unordered pair $(X,Y)$ of sets of real numbers {\em interlaces} if (i) the number of elements in $X$ and $Y$ are equal or differ by one,
and (ii) there exists an ordering 
$\alpha_1\le\beta_1\le\alpha_2\le\beta_2\le\cdots \le\alpha_\nu\le\beta_\nu\le\cdots$,
where $\alpha_1$, $\alpha_2$, $\ldots$ are the elements of one of the sets and $\beta_1$, $\beta_2$, $\ldots$ are those of the other; and {\em strictly interlaces} if no equality in the foregoing ordering holds.
Two real-rooted polynomials $P$ and $Q$ are said to {\em have interlacing zeros} 
if the pair of their zero sets interlaces.
A degenerated case is that the pair of any singleton and the empty set~$\emptyset$ interlaces, and so does the pair $(\emptyset,\, \emptyset)$.
\end{dfn}

We introduce the concept of piecewise interlacing pairs. 

\begin{dfn}
We say a pair $X$ and $Y$ of sets of real numbers \emph{piecewise interlaces} 
(resp.\ \emph{piecewise strictly interlaces}) 
on the disjoint union $\sqcup_{\lambda\in\Lambda} I_\lambda$ of intervals,
if the subset pair $(X\cap I_\lambda,\,Y\cap I_\lambda)$ for each $\lambda\in\Lambda$ interlaces 
(resp.\ strictly interlaces). 
\end{dfn}

Suppose that we are going to show the real-rootedness of 
every polynomial in a polynomial sequence $\{W_n(z)\}_n$,
or to study the relation of root distribution of consecutive polynomials $W_n(z)$ and $W_{n+1}(z)$.
With the aid of the concept of piecewise interlacing zeros, we can do this in two steps.
First, we deal with repeated zeros of each $W_n(z)$, and with common zeros of different polynomials $W_n(z)$,
individually. This step may be done in specific ways; see~\cite{JW17X} for instance.
Then, secondly, we can suppose that the polynomials $W_n(z)$ have neither common zeros nor repeated zeros.

The basic idea of confirming the real-rootedness is to show that consecutive polynomials 
have piecewise interlacing zeros by induction using the intermediate value theorem.
The difficulty lies in the division of the real line into disjoint intervals so that an induction proof works.
More precisely, the points that one would like to assign to be an end of the desired intervals 
often obey a regular law of sign changing in the subscript $n$. 
For example, traditionally one considers the whole real line
on which polynomials are expected to have interlacing zeros, that is, 
to take the interval ends to be $-\infty$ and $+\infty$.
That is partially because that the signs of $\lim_{x\to -\infty}W_n(x)$ and $\lim_{x\to \infty}W_n(x)$ 
for each polynomial $W_n(x)$ in real~$x$ can be simply determined by 
the sign of leading coefficient of $W_n(x)$ and the parity of the degree of $W_n(x)$.
It is usually easy to figure out the sign of leading coefficient and the parity of the degree for each $W_n(x)$
from a recursive definition.

Of course it is possible that two polynomials may not have interlacing zeros over the whole real line.
It is not uncommon that the intervals are well hidden behind the appearance 
of the given polynomial sequence.
Besides the points $\pm\infty$, one may also select ``isolated limits of zeros'' to be interval ends;
see Beraha, Kahane and Weiss~\cite{BKW75,BKW78} and Sokal~\cite{Sok04} for definition. 
We will illustrate this idea in the next section.

In fact, there has been a number of study on the limiting root distribution of polynomials 
in a sequence; see Szeg\H o~\cite{Sze24} for instance. Stanley on his website \cite{StaW}
provides some figures for the root distribution of some polynomials in a sequence 
arising from combinatorics. 
Bleher and Mallison~\cite{BM06} studied the asymptotics and zeros of Taylor polynomials
for linear combinations of exponentials.
In answering a problem posed by Herbert Wilf,
Boyer and Goh~\cite{BG07} figured out that the set of limit points of zeros of Euler polynomials,
called its ``zero attractor'',  to be the union of an interval and a curve closely related to the Szeg\H o's curve;
See also~\cite{BG08,GHR09}.

\section{An application of the piecewise interlacing zeros}\label{sec:main}

Let $a,b,c,d\in\R$ and $ac\ne0$. 
In this paper, we consider polynomials $W_n(z)$ satisfying the recurrence
\begin{equation}\label[rec]{rec11}
W_n(z) = (az+b)W_{n-1}(z)+(cz+d)W_{n-2}(z). 
\end{equation}
We call the sequence $\{W_n(z)\}_{n\ge0}$ \emph{normalised} if $W_0(z)=1$ and $W_1(z)=z$. 
For any complex number $z=re^{i\theta}$ with $\theta\in(-\pi,\,\pi]$,
we use the square root notation $\sqrt{z}$ to denote 
the number $\sqrt{r}e^{i\theta/2}$.
\Cref{lem:00}
is the base of our study, which can be found in~\cite{KL93B,GMTW16-01,GMTW16-10}. 

\begin{lem}\label{lem:00}
Let $A,B\in\C$.
Suppose that $W_0=1$ and $W_n=AW_{n-1}+BW_{n-2}$ for $n\ge 2$. Then for $n\ge 0$, we have 
\[
W_n=\begin{cases}
{\displaystyle {\alpha_+\lambda_+^n+\alpha_-\lambda_-^n}},&\textrm{ if $\Delta\neq0$},\\[5pt]
{\displaystyle {A+n(2W_1-A)\over 2}\cdot\bgg{{A\over 2}}^{n-1}},&\textrm{ if $\Delta=0$},
\end{cases}
\]
where $\lambda_\pm=(A\pm\sqrt{\Delta})/2$ with $\Delta=A^2+4B$
are the eigenvalues, and 
\[
\alpha_\pm=\frac{\sqrt{\Delta}\pm(2W_1-A)}{2\sqrt{\Delta}}. 
\] 
\end{lem}

We employ the notations
\begin{align*}
\Delta(z)&=A(z)^2+4B(z)
=a^2z^2+(2ab+4c)z+(b^2+4d),\\[4pt]
\alpha_\pm(z)&=\frac{\sqrt{\Delta(z)}\pm(2W_1(z)-A(z))}{2\sqrt{\Delta(z)}},\rmand\\[4pt]
g(z)&=-\alpha_+(z)\alpha_-(z)\Delta(z)=(1-a)z^2-(b+c)z-d,
\end{align*}
whose zeros are respectively
\begin{equation}\label{xd}
x_\Delta^{\pm}=\frac{-ab-2c\pm2\sqrt{\Delta_\Delta}}{a^2}
\end{equation}
where $\Delta_\Delta=-a^2d+abc+c^2$, and
\begin{equation}\label{xg}
x_g^\pm=\begin{cases}
\displaystyle
\frac{b+c}{2(1-a)}\pm\frac{\sqrt{\Delta_g}}{2\abs{1-a}},&\text{if $a\ne1$},\\[8pt]
\displaystyle -{d\over b+c},&\text{if $a=1$ and $b+c\ne0$},
\end{cases}
\end{equation}
where  $\Delta_g=(b+c)^2+4d(1-a)$.
In fact, the function $g(z)=-d$ is a constant when $a=1$ and $b+c=0$.
For convenience, we also use  
\[
x_A=-b/a\rmand x_B=-d/c.
\]
From \cref{rec11}, it is routine to check that $W_n(z)$ has leading coefficient $a^{n-1}z^n$, and that 
\begin{align}
W_n(x_g^\pm)&=(x_g^\pm)^n,\label{W:xg}\\[5pt]
W_n(x_\Delta^\pm)
&=\frac{A(x_\Delta^\pm)+nh(x_\Delta^\pm)}{2}\cdot\bgg{\frac{A(x_\Delta^\pm)}{2}}^{n-1},\label{W:xd}
\end{align}
where $h(z)=2W_1(z)-A(z)=(2-a)z-b$. 

The polynomials $W_n(z)$ in \cref{lem:00} satisfy another recurrence as
presented in \cref{lem:rec2}, which helps in proving \cref{thm:rr,thm:pl}.

\begin{lem}\label{lem:rec2}
Let $A,B\in\R\setminus\{0\}$.
Suppose that $W_0=1$ and $W_n=AW_{n-1}+BW_{n-2}$ for $n\ge 2$. 
Then $W_n=(A^2+2B)W_{n-2}-B^2W_{n-4}$ for $n\ge4$.
\end{lem}

To the end of this paper, let $\{W_n(z)\}_n$ be a \noo\ with $a,b,d<0$ and $c>0$. 
We introduce the notation 
\[
c^\pm=\pm2\sqrt{d(a-1)}-b,
\]
as the zeros of the discriminant~$\Delta_g$ in $c$. It is clear that 
\[
\max(0,\,c^-)<c^+.
\] 
From definition, we have 
\[
x_g^\pm\in\R\qquad\iff\qquad 
c\le c^-\quad\text{or}\quad c\ge c^+.
\]
In this case, we denote 
\[
J_g=(x_g^-,\,x_g^+).
\]
Let $R_n$ be the zero set of~$W_n(z)$. 
It is a multi-set when $W_n(z)$ has repeated zeros,
though all zeros considered in this paper turn out to be simple eventually. 
Denote 
\[
R_n\cap J=R_n^J\qquad\text{for any interval $J$}. 
\]
For the statement of \cref{thm:rr}, we denote 
\[
n^+=-\frac{A(x_\Delta^+)}{h(x_\Delta^+)}
\qquad\text{whenever $h(x_\Delta^+)\ne0$}.
\]

\begin{thm}\label{thm:rr}
Let $\{W_n(z)\}_n$ be a \noo\ with $a,b,d<0$ and $c>0$. 
Define 
\[
J_1=(x_\Delta^-,\,x_A),\qquad
J_2=(x_A,\,x_\Delta^+),\qquad
J_3=[\,x_\Delta^+,\,x_g^-),\qquad
J_4=(x_g^+,\,0\,].  
\]
\begin{itemize}
\itemsep 5pt
\item[(i)]
If $c\le c^-$, then every polynomial $W_n(z)$ is real-rooted, 
\[
\babs{R_n^{J_1}}=\fl{(n-1)/2},\qquad
\babs{R_n^{J_2\cup J_3}}=\fl{n/2},\rmand
\babs{R_n^{J_4}}=1.
\]
Moreover, we have
\[
\babs{R_n^{J_3}}=\begin{cases}
1,&\text{if $\Delta_\Delta>\Delta_g$ and $n\ge n^+$};\\[3pt]
0,&\text{otherwise}.
\end{cases}
\]
\item[(ii)]
If $c\ge c^+$, then every polynomial $W_{2n}(z)$ is real-rooted, and 
\[
\babs{R_{2n}^{J_1}}=\babs{R_{2n}^{J_2}}=n-1\rmand
\babs{R_{2n}^{J_3}}=\babs{R_{2n}^{J_4}}=1. 
\]
\item[(iii)]
If $c>c^+$ and the polynomial $W_3(z)$ has zeros $\xi_{3,2}\le \xi_{3,3}$ lying in the interval~$J_g$,
then every polynomial $W_{2n+1}(z)$ is real-rooted, and
\[
\babs{R_{2n+1}^{J_1}}=n-1,\qquad
\babs{R_{2n+1}^{J_2}}=n,\rmand
\babs{R_{2n+1}^{I_3}}=\babs{R_{2n+1}^{I_4}}=1,
\]
where $I_3=(x_g^-,\,\xi_{3,2})$ and $I_4=(\xi_{3,3},\,x_g^+)$ for $n\ge 2$. 
\end{itemize}
\end{thm}

We illustrate \cref{thm:rr} as in \cref{fig1,fig11,fig2,fig3} respectively. 
In each figure, we thicken the line segment $J_\Delta$ and the points $x_g^\pm$.
The numbers of zeros are indicated above corresponding intervals.

\begin{figure}[htbp]
\begin{tikzpicture}
\draw[->] (0,0) -- (13,0) coordinate (x axis);

\draw[ultra thick] (1 cm,2pt) -- (1 cm,-2pt);
\draw[ultra thick] (3 cm,2pt) -- (3 cm,-2pt);
\draw[ultra thick] (5 cm,2pt) -- (5 cm,-2pt);
\draw[ultra thick] (11 cm,2pt) -- (11 cm,-2pt);
\fill (7,0) circle (3pt);
\fill (9,0) circle (3pt);
\draw[ultra thick] (1,0) -- node[below=5pt] {$J_1$} (3,0);
\draw[ultra thick] (3,0) -- node[below=5pt] {$J_2$} (5,0);
\draw (6,0) node[below=5pt] {$J_3$};
\draw (10,0) node[below=5pt] {$J_4$};

\draw (1,0) node[below=3pt] {$x_\Delta^-$};
\draw (3,0) node[below=7pt] {$x_A$};
\draw (5,0) node[below=3pt] {$x_\Delta^+$};
\draw (7,0) node[below=3pt] {$x_g^-$};
\draw (9,0) node[below=3pt] {$x_g^+$};
\draw (11,0) node[below=7pt] {$0$};

\draw (1,0) -- node[above=5pt] {$\bfl{\frac{n-1}{2}}$} (3,0);
\draw (3,0) -- node[above=5pt] {$\bfl{n\over 2}-1$} (5,0);
\draw (5,0) -- node[above=5pt] {$1$} (8,0);
\draw (9,0) -- node[above left = 5pt] {$1$} (11,0);
\end{tikzpicture}
\caption{Illustration of the root distribution in \cref{thm:rr} (i) 
when $\Delta_\Delta>\Delta_g$ and $n\ge n^+$.}\label{fig1}
\end{figure}
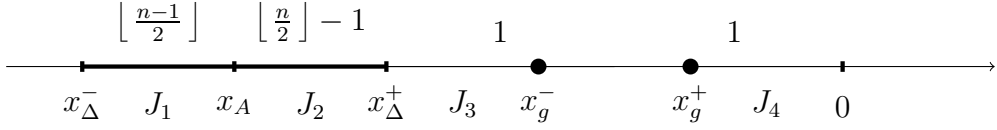

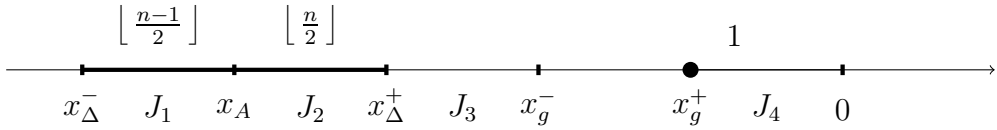
\begin{figure}[htbp]
\begin{tikzpicture}
\draw[->] (0,0) -- (13,0) coordinate (x axis);

\draw[ultra thick] (1 cm,2pt) -- (1 cm,-2pt);
\draw[ultra thick] (3 cm,2pt) -- (3 cm,-2pt);
\draw[ultra thick] (5 cm,2pt) -- (5 cm,-2pt);
\draw[ultra thick] (11 cm,2pt) -- (11 cm,-2pt);
\draw[ultra thick] (7 cm,2pt) -- (7 cm,-2pt);
\fill (9,0) circle (3pt);
\draw[ultra thick] (1,0) -- node[below=5pt] {$J_1$} (3,0);
\draw[ultra thick] (3,0) -- node[below=5pt] {$J_2$} (5,0);
\draw (6,0) node[below=5pt] {$J_3$};
\draw (10,0) node[below=5pt] {$J_4$};

\draw (1,0) node[below=3pt] {$x_\Delta^-$};
\draw (3,0) node[below=7pt] {$x_A$};
\draw (5,0) node[below=3pt] {$x_\Delta^+$};
\draw (7,0) node[below=3pt] {$x_g^-$};
\draw (9,0) node[below=3pt] {$x_g^+$};
\draw (11,0) node[below=7pt] {$0$};

\draw (1,0) -- node[above=5pt] {$\bfl{\frac{n-1}{2}}$} (3,0);
\draw (3,0) -- node[above=5pt] {$\bfl{\frac{n}{2}}$} (5,0);
\draw (9,0) -- node[above left = 5pt] {$1$} (11,0);
\end{tikzpicture}
\caption{Illustration of the root distribution in \cref{thm:rr} (i) 
when either $\Delta_\Delta>\Delta_g$ and $n<n^+$, or $\Delta_\Delta\le\Delta_g$.}\label{fig11}
\end{figure}

\bigskip

\begin{figure}[htbp]
\begin{tikzpicture}
\draw[->] (0,0) -- (14,0) coordinate (x axis);
\draw[ultra thick] (1 cm,2pt) -- (1 cm,-2pt);
\draw[ultra thick] (3 cm,2pt) -- (3 cm,-2pt);
\draw[ultra thick] (5 cm,2pt) -- (5 cm,-2pt);
\draw[ultra thick] (7 cm,2pt) -- (7 cm,-2pt);
\draw[ultra thick] (13 cm,2pt) -- (13 cm,-2pt);
\fill (9,0) circle (3pt);
\fill (11,0) circle (3pt);
\draw[ultra thick] (1,0) -- node[below=5pt] {$J_1$} (3,0);
\draw[ultra thick] (3,0) -- node[below=5pt] {$J_2$} (5,0);
\draw (7,0) -- node[below=5pt] {$J_3$} (9,0);
\draw (11,0) -- node[below=5pt] {$J_4$} (13,0);

\draw (1,0) node[below=3pt] {$x_\Delta^-$};
\draw (3,0) node[below=7pt] {$x_A$};
\draw (5,0) node[below=3pt] {$x_\Delta^+$};
\draw (7,0) node[below=7pt] {$x_B$};
\draw (9,0) node[below=3pt] {$x_g^-$};
\draw (11,0) node[below=3pt] {$x_g^+$};
\draw (13,0) node[below=7pt] {$0$};

\draw (1,0) -- node[above=5pt] {$n-1$} (3,0);
\draw (3,0) -- node[above=5pt] {$n-1$} (5,0);
\draw (7,0) -- node[above right = 5pt] {$1$} (9,0);
\draw (11,0) -- node[above left = 5pt] {$1$} (13,0);
\end{tikzpicture}
\caption{Illustration of the root distribution in \cref{thm:rr} (ii).}\label{fig2}
\end{figure}

\bigskip

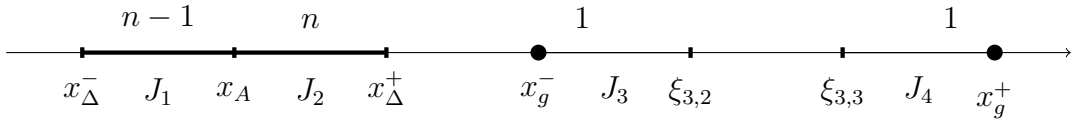
\begin{figure}[htbp]
\begin{tikzpicture}
\draw[->] (0,0) -- (14,0) coordinate (x axis);
\draw[ultra thick] (1 cm,2pt) -- (1 cm,-2pt);
\draw[ultra thick] (3 cm,2pt) -- (3 cm,-2pt);
\draw[ultra thick] (5 cm,2pt) -- (5 cm,-2pt);
\draw[ultra thick] (9 cm,2pt) -- (9 cm,-2pt);
\draw[ultra thick] (11 cm,2pt) -- (11 cm,-2pt);
\fill (7,0) circle (3pt);
\fill (13,0) circle (3pt);
\draw[ultra thick] (1,0) -- node[below=5pt] {$J_1$} (3,0);
\draw[ultra thick] (3,0) -- node[below=5pt] {$J_2$} (5,0);
\draw (7,0) -- node[below=5pt] {$J_3$} (9,0);
\draw(11,0) -- node[below=5pt] {$J_4$} (13,0);

\draw (1,0) node[below=3pt] {$x_\Delta^-$};
\draw (3,0) node[below=7pt] {$x_A$};
\draw (5,0) node[below=3pt] {$x_\Delta^+$};
\draw (7,0) node[below=3pt] {$x_g^-$};
\draw (9,0) node[below=5pt] {$\xi_{3,2}$};
\draw (11,0) node[below=5pt] {$\xi_{3,3}$};
\draw (13,0) node[below=7pt] {$x_g^+$};

\draw (1,0) -- node[above=5pt] {$n-1$} (3,0);
\draw (3,0) -- node[above=5pt] {$n$} (5,0);
\draw (7,0) -- node[above left = 5pt] {$1$} (9,0);
\draw (11,0) -- node[above right = 5pt] {$1$} (13,0);
\end{tikzpicture}
\caption{Illustration of the root distribution in \cref{thm:rr} (iii).}\label{fig3}
\end{figure}

Below are some examples as an application of \cref{thm:rr}.
\begin{eg}\label{eg1}
Let $c\in\R$. Let $W_n(z)$ be polynomials defined by the recurrence
\[
W_n(z)=(-3z-5)W_{n-1}(z)+(cz-1)W_{n-2}(z),
\]
with $W_0(z)=1$ and $W_1(z)=z$. It is routine to compute that
\begin{align*}
W_2(z)&=-3z^2+(c-5)z-1,\\[4pt]
W_3(z)&=9z^3+(30-2c)z^2+(27-5c)z+5,\\[4pt]
W_4(z)&=-27z^4-(3c-135)z^3+(c^2+20c-228)z^2+(23c-145)z-24,\\[4pt]
W_5(z)&=81z^5-(5c^2+45c-1350)z^4+(212-3c^2)z^3\\[3pt]
&\quad -(10c^2+140c-1545)z^2+(770-105c)z+115.
\end{align*}
From definition, one may compute that $c^-=1$ and $c^+=9$. 
\begin{itemize}
\itemsep 7pt
\item
If $c=0.8$, then \cref{thm:rr} (i) implies the real-rootedness of every polynomial $W_n(z)$. For example, $W_4(z)$ has the real-zeros 
\[
\xi_{4,1}\approx-2.396, \qquad
\xi_{4,2}\approx-1.446, \qquad
\xi_{4,3}\approx-0.704, \qquad
\xi_{4,4}\approx-0.364.
\]

\item
If $c=10$, then every polynomial $W_{2n}(z)$ is real-rooted by \cref{thm:rr} (ii). 
In this case, we can compute that
\[
x_g^-=0.25,\qquad 
\xi_{3,2}\approx0.251,\qquad 
\xi_{3,3}\approx0.955,\qquad
x_g^+=1. 
\]
Thus $x_g^-<\xi_{3,2}<\xi_{3,3}<x_g^+$.
By \cref{thm:rr} (iii), every polynomial $W_{2n-1}(z)$ is also real-rooted. 
For example, the zeros of $W_5(z)$ have the following approximations.
\[
\xi_{5,1}\approx-15.70, \quad
\xi_{5,2}\approx-1.962, \quad
\xi_{5,3}\approx-0.534, \quad
\xi_{5,4}\approx0.250, \quad
\xi_{5,5}\approx0.999.
\]   
\end{itemize}
\end{eg}
In fact, when $c^->0$, then the real-rootedness of the polynomial $W_3(z)$ 
implies that $\babs{R_3^{J_g}}=2$.  This can be seen from \cref{thm:W3:--+-}. 
When $c^-<0$, nevertheless, it is possible that $\babs{R_3^{J_g}}=2$. Below is such an example. 

\begin{eg}\label{eg2}
Let $c\in\R$. Let $W_n(z)$ be polynomials defined by the recurrence
\[
W_n(z)=(-0.3z-1)W_{n-1}(z)+(cz-60)W_{n-2}(z),
\]
with $W_0(z)=1$ and $W_1(z)=z$.
From definition, one may compute that 
$c^-\approx -16.6$ and $c^+\approx 18.6$. 
If $c=65$, then 
\[
x_g^-\approx 0.956,\qquad 
\xi_{3,2}\approx1.014,\qquad 
\xi_{3,3}\approx1.276,\qquad
x_g^+\approx48.274.
\]
Thus $x_g^-<\xi_{3,2}<\xi_{3,3}<x_g^+$.
By \cref{thm:rr}, every polynomial $W_{n}(z)$ is real-rooted. 
For example, the zeros of $W_5(z)$ have the following approximations.
\[
\xi_{5,1}\approx-1844.053, \quad
\xi_{5,2}\approx -1255.040, \quad
\xi_{5,3}\approx0.912, \quad
\xi_{5,4}\approx0.958, \quad
\xi_{5,5}\approx4.352.
\]
\end{eg}

Before ending this section,
we explain how one obtains the intervals $J_i$ in \cref{thm:rr}.
According to Beraha et al.'s result \cite{BKW78}, one may
calculate all limits of zeros of the polynomials. 
\emph{Non-isolated limits of zeros} depend only on \cref{rec11},
while those \emph{isolated} relies on the initial polynomials additionally,
that is $W_0(z)=1$ and $W_1(z)=z$ in the normalised case.
In fact, the non-isolated limits have bounds $x_\Delta^-$ and $x_\Delta^+$,
and the set of isolated limits consists of the points $x_g^\pm$.

Now we have the four cutting points $x_\Delta^\pm$ and $x_g^\pm$ in hand.
They work already well when the coefficients $A(z)$ and $B(z)$ have lower degrees,
see \cite{GMTW16-01,GMTW16-10} for when one of them is a constant. 
The coefficients in this paper have both degree one, 
and we adopt a fifth point $x_A$ to be an end of the desired intervals.
In fact, the point~$x_A$ is a non-isolated limit of zeros as long as the coefficient $A(z)$ is of degree one, 
and it is easy to determine the sign of $W_n(x_A)$ for each $n$ 
by \cref{rec11}. 
In general, the increasing order prompts us to introduce more 
cutting points to divide the real line.

\section{Proofs of the main result}\label[sec]{sec:pf}

The section contains a proof of \cref{thm:rr}, 
a result on the root distribution of the polynomial~$W_3(z)$ with a proof.

In \cref{--+-:lem:basic},
we list some signs that will be used in the sequel.

\begin{lem}\label{--+-:lem:basic}
Let $\{W_n(z)\}_{n\ge0}$ be a \noo\
with $a,b,d<0$ and $c>0$.
Then we have 
\[
x_\Delta^-<x_A<0<x_B
\rmand 
x_A<x_\Delta^+<x_B.
\]
It follows that $A(x_\Delta^-)>0>A(x_\Delta^+)$.
For $n\ge1$, we have $(-1)^{\cl{n/2}}W_n(x_A)>0$, $(-1)^nW_n(x_B)<0$, $W_n(x_\Delta^-)<0$, and the following. 
\begin{itemize}
\itemsep 5pt
\item[(i)]
If $c\le c^-$, then $x_\Delta^+\le x_g^-\le x_g^+<0$, $W_n(x_g^\pm)(-1)^n>0$, and
\[
(-1)^nW_n(x_\Delta^+)\begin{cases}
<0,&\text{if $\Delta_\Delta>\Delta_g$ and $n>n^+$;}\\[3pt]
=0,&\text{if $\Delta_\Delta>\Delta_g$ and $n=n^+$;}\\[3pt]
>0,&\text{otherwise}.
\end{cases}
\]
\item[(ii)]
If $c\ge c^+$, then $x_B<x_g^-$, $W_n(x_g^\pm)>0$, and
$(-1)^nW_n(x_\Delta^+)<0$ for $n\ge 2$. 
\end{itemize}
\end{lem}

\begin{proof}
The order relations among the numbers $x_\Delta^\pm$, $x_A$ and $x_B$
can be derived directly from definition. 
From the linearity of the function $A(z)$ and that of $h(z)$, we find 
$A(x_\Delta^-)>0>A(x_\Delta^+)$ and $h(x_\Delta^-)<0$.
From \cref{rec11}, we can deduce that
\begin{align*}
W_n(x_A)&=\begin{cases}
\bg{c(x_A-x_B)}^{n/2},&\text{if $n$ is even},\\[9pt]
\bg{c(x_A-x_B)}^{(n-1)/2}x_A,&\text{if $n$ is odd},
\end{cases}\\[5pt]
(-1)^nW_n(x_B)&=(-1)^n\bg{a(x_B-x_A)}^{n-1}x_B<0. 
\end{align*}
On the other hand, we note that
\[
A(x_\Delta^-)+n\cdot h(x_\Delta^-)
=2x_\Delta^-+(n-1)h(x_\Delta^-)<0.
\]
By \cref{W:xd}, we infer that $W_n(x_\Delta^-)<0$. 
We proceed according to the range of $c$.

\begin{case}[$c\le c^-$]
In this case, the negativities of $x_g^\pm$ can be shown
by using Vi\`eta's formula to the polynomial $g(z)$. 
It follows from \cref{W:xg} that $W_n(x_g^\pm)(-1)^n>0$. 
By \cref{xg,xd}, it is routine to verify that
\[
x_g^--x_\Delta^+
=\frac{2(\sqrt{\Delta_\Delta}-\sqrt{\Delta_g})^2}
{a^2\sqrt{\Delta_g}+4(1-a)\sqrt{\Delta_\Delta}-a^2b+a^2c+2ab-4ac+4c}\ge 0.
\]
Equally routine can we check that
\[
h(x_\Delta^+)=\frac{2(\Delta_\Delta-\Delta_g)}{ab-ac+2c+(2-a)\sqrt{\Delta_\Delta}},
\]
which implies that the values $h(x_\Delta^+)$ and $\Delta_\Delta-\Delta_g$ have the same sign.
The sign of $W_n(x_\Delta^+)$ can be determined by \cref{W:xd}. 
\end{case}

\begin{case}[$c\ge c^+$] 
Vi\`eta's formula gives the positivities of~$x_g^\pm$. Then $W_n(x_g^\pm)>0$. 
In this case, we have
\begin{align*}
c(b+c)+2d(1-a)
&\ge2\sqrt{d(a-1)}(c-\sqrt{d(a-1)})\\
&\ge2\sqrt{d(a-1)}(\sqrt{d(a-1)}-b)>0.
\end{align*}
Thus
\[
x_g^--x_B=\frac{2x_B\bg{(a-1)d-bc}}
{c\sqrt{\Delta_g}+c(b+c)+2d(1-a)}>0.
\]
By using the condition $c\ge c^+$, 
it is routine to verify that
$\Delta_g<\Delta_\Delta$. 
Denote by $z_h=b/(a-2)$ the zero of $h(z)$. 
Then $\Delta(z_h)=4(\Delta_g-\Delta_\Delta)/(2-a)^2<0$.
It follows that $z_h<x_\Delta^+$, $h(x_\Delta^+)>0$ and thus $n^+>0$.
By \cref{W:xd}, we can deduce that 
\[
(-1)^nW_n(x_\Delta^+)<0
\iff
n>n^+.
\]
It remains to show $n^+<2$.
It is routine to check that $n^+=N^+/D^+$, where 
\[
N^+=a(\sqrt{\Delta_\Delta}-c)<0\rmand
D^+=(a-2)(\sqrt{\Delta_\Delta}-c)+ab.
\]
Thus $D^+<0$, and the desired inequality $n^+<2$ holds 
if and only if $N^+>2D^+$,
which simplifies to $R>L$, where
\begin{align*}
L=2ab-ac+4c>0
\rmand
R=(4-a)\sqrt{\Delta_\Delta}>0.
\end{align*}
The remaining proof follows immediately from the fact $R^2>L^2$.
\end{case}
\vskip -10pt
This completes the proof.
\end{proof}

\Cref{thm:pl} is the aforementioned result that the polynomials $W_n(z)$
have triple piecewise interlacing zeros. 

\begin{thm}[Piecewise interlacing zeros]\label{thm:pl}
Let $\{W_n(z)\}_n$ be a \noo\
with $a,b,d<0$ and $c>0$.
\begin{itemize}
\itemsep 5pt
\item[(i)]
If $0<c\le c^-$, then the pair $(R_{n+1},\,R_{n})$ piecewise strictly interlaces
on the union $\sqcup_{j=1}^4 J_j$,
and so do the pairs $(R_{2n+2},\,R_{2n})$ and $(R_{2n+1},\,R_{2n-1})$. 
As $n\to\infty$, the possible zero in $R_n^{J_3}$ increases, and the zero in $R_n^{J_4}$ decreases. 
\item[(ii)]
If $c\ge c^+$, then the pair $(R_{2n+2},\,R_{2n})$ piecewise strictly interlaces 
on the union $\sqcup_{j=1}^4 J_j$. 
As $n\to\infty$, the zero in $R_{2n}^{J_3}$ increases and the zero in $R_{2n}^{J_4}$ decreases. 
\item[(iii)]
If $c>c^+$ and $W_3(z)$ has zeros $\xi_{3,2}\le \xi_{3,3}$ lying in $J_g$,
then the pair $(R_{n+1},\,R_n)$ piecewise strictly interlaces 
on the union $\sqcup_{j=1}^4 J_j$, and so does the pair $(R_{2n+1},R_{2n-1})$. 
As $n\to\infty$, the zero in $R_{2n+1}^{J_3}$ decreases 
and the zero in~$R_{2n+1}^{J_4}$ increases. 
\end{itemize}
\end{thm}

We shall prove \cref{thm:rr,thm:pl} together. 
From \cref{--+-:lem:basic}, we see that $x_\Delta^-<x_A<x_g^-\le x_g^+<0$ 
when $c\le c^-$.

\noindent{\bf Proof of \cref{thm:rr} (i) and \cref{thm:pl} (i).}
First of all, we show the results in \cref{thm:rr} (i) and \cref{thm:pl} (i) except that for the cardinality 
$\babs{R_n^{J_3}}$.

From \cref{--+-:lem:basic}, we see that $B(x)<0$ for all negative $x$.
Before proceeding by induction, we check the desired results for $n\le 4$.

\noindent{\textbullet\mathversion{bold} $n=1$}.
The polynomial $W_1(z)=z$ has the unique zero $0\in J_4$.

\noindent{\textbullet\mathversion{bold} $n=2$}.
It is direct to compute that $W_2(0)=d<0$.
By the intermediate value theorem, 
with aid of \cref{--+-:lem:basic}, 
the polynomial $W_2(x)$ has zeros 
\[
\xi_{2,1}\in J_2\cup J_3
\rmand
\xi_{2,2}\in J_4.
\]
\noindent{\textbullet\mathversion{bold} $n=3$}.
Since $\xi_{2,j}<0$ for $j\in[2]$,
we infer that $B(\xi_{2,j})<0$.
Since $W_2(\xi_{2,j})=0$, we deduce from \cref{rec11} that 
\[
W_3(\xi_{2,j})=B(\xi_{2,j})W_1(\xi_{2,j})
=B(\xi_{2,j})\xi_{2,j}>0.
\]
For the same reason, the polynomial $W_3(x)$ has zeros 
\[
\xi_{3,1}\in J_1,\qquad
\xi_{3,2}\in(\xi_{2,1},\,x_g^-)\subset J_2\cup J_3,\rmand
\xi_{3,3}\in(x_g^+,\,\xi_{2,2})\subset J_4.
\]
\noindent\textbullet{\mathversion{bold} $n=4$}.
By \cref{lem:rec2} and the facts $W_2(\xi_{2,1})=0$ and $W_0(z)=1$, 
we infer that
\[
W_4(\xi_{2,1})=-B^2(\xi_{2,1})W_0(\xi_{2,1})=-B^2(\xi_{2,1})<0.
\]
Since $\xi_{3,j}<0$ for $j\in[3]$,
we obtain $B(\xi_{3,j})<0$.
By \cref{rec11} and together with the facts $W_3(\xi_{3,j})=0$ and 
\[
\sgn\bg{W_2(\xi_{3,j})}=\begin{cases}
-1,&\text{if $j=1$};\\[3pt]
1,&\text{if $j=2,3$},
\end{cases}
\]
we infer that 
\[
\sgn\bg{W_4(\xi_{3,j})}
=\sgn\bg{B(\xi_{3,j})W_2(\xi_{3,j})}
=\begin{cases}
1,&\text{if $j=1$};\\[3pt]
-1,&\text{if $j=2,3$}.
\end{cases}
\]
Again, we can deduce that $W_4(z)$ has a zero each of the following intervals:
\[
(x_\Delta^-,\,\xi_{3,1})\subset J_1,\quad
(x_A,\,\xi_{2,1})\subset J_2\cup J_3,\quad
(\xi_{3,2},\,x_g^-)\subset J_2\cup J_3,\quad\text{and}\quad
(x_g^+,\,\xi_{3,3})\subset J_4.
\]

Let {\mathversion{bold} $n\ge5$}. We proceed by induction. 
Suppose that the polynomial $W_{n-2}(z)$ has zeros $x_1<\cdots<x_{n-2}$,
and that the polynomial $W_{n-1}(z)$ has zeros $y_1<\cdots<y_{n-1}$.
The induction hypotheses on the interlacing zeros give us immediately
\[
\sgn\bg{W_{n-4}(x_j)}=\begin{cases}
(-1)^j,&\text{if $j\in\fl{(n-3)/2}$};\\
(-1)^{j+1},&\text{if $\fl{(n-1)/2}\le j\le n-3$}.
\end{cases}
\]
Note that $B(x_j)\ne0$.
By \cref{lem:rec2} and the fact $W_{n-2}(x_j)=0$, 
we infer that
\[
\sgn\bg{W_n(x_j)}
=\sgn\bg{-B^2(x_j)W_{n-4}(x_j)}
=\begin{cases}
(-1)^{j+1},&\text{if $j\in\fl{(n-3)/2}$};\\
(-1)^j,&\text{if $\fl{(n-1)/2}\le j\le n-3$}.
\end{cases}
\]
On the other hand,
by \cref{rec11}, the facts $W_{n-1}(y_{n-1})=0$ and
$B(y_{n-1})<0$, and the interlacing zeros by induction, we deduce that
\begin{align*}
W_{n-2}(x_g^+)W_n(y_{n-1})
&=W_{n-2}(x_g^+)\cdot \bg{B(y_{n-1})W_{n-2}(y_{n-1})}\\
&=B(y_{n-1})\cdot \bg{W_{n-2}(x_g^+)W_{n-2}(y_{n-1})}
<0.
\end{align*}
By the intermediate value theorem
and with the aid of \cref{--+-:lem:basic},
the polynomial $W_n(z)$ has zeros $z_j$ ($j\in[n]$) such that
\begin{multline*}
x_\Delta^-
<z_1<x_1<z_2<x_2
<\cdots
<x_{\fl{(n-3)/2}}<z_{\fl{(n-1)/2}}<x_A\\
<z_{\fl{(n+1)/2}}<x_{\fl{(n-1)/2}}
<z_{\fl{(n+3)/2}}<x_{\fl{(n+1)/2}}
<\cdots<x_{n-3}<z_{n-1}<x_g^-\\
<x_g^+<z_n<y_{n-1}<x_{n-2}<\cdots\le0.
\end{multline*}
Since $\deg W_n(z)=n$,
we conclude that the zeros $z_j$ ($j\in[n]$)
constitute the zero set of $W_n(z)$.
This proves the real-rootedness of $W_n(z)$,
verifies the cardinalities $\babs{R_n^{J_1}}$, $\babs{R_n^{J_2\cup J_3}}$, and $\babs{R_n^{J_4}}$, 
and establishes the desired interlacing properties 
of the sequences.
By induction, we have $W_{n-2}(y_j)(-1)^j>0$ for $j\in[n-2]$.
By \cref{rec11} and the facts $W_{n-1}(y_j)=0$ and $B(y_j)<0$, 
we derive
\[
W_n(y_j)(-1)^j=B(y_j)W_{n-2}(y_j)(-1)^j<0.
\]
For the same reason, we obtain that $W_n(z)$ and $W_{n-1}(z)$ have interlacing zeros
over each of the intervals $J_j$. 

Now we are going to determine the cardinality $\babs{R_n^{J_3}}$.

When $\Delta_\Delta\le\Delta_g$, from \cref{--+-:lem:basic}, 
we see that the numbers $W_n(x_\Delta^+)$ and $W_n(x_g^-)$ have the same sign.
This leads us to replace the number $x_g^-$ by $x_\Delta^+$ in the foregoing proof, 
and the interval $J_2\cup J_3$ becomes $J_2$.
The whole deduction remains true, 
which changes the conclusion $\babs{R_n^{J_2\cup J_3}}=\fl{n/2}$ to $\babs{R_n^{J_2}}=\fl{n/2}$.
Thus $R_n^{J_3}=\emptyset$ as desired.

When $\Delta_\Delta>\Delta_g$, 
the desired cardinality can be determined in a way 
highly similar to the proof of Theorem 4.5 in \cite{GMTW16-10}.
We provide a proof sketch for completeness.
In fact, we can deduce that $h(x_\Delta^+)> 0$, $n^+>0$, and
$W_n(x_\Delta^+)W_n(x_g^-)(n-n^+)\le0$,
with the equality holds if and only if $n=n^+$. 
The interlacing property over the interval $J_2\cup J_3$ implies
$\babs{R_{n+1}^{J_3}}-\babs{R_n^{J_3}}\le 1$.
By bootstrapping and with the aid of the intermediate value theorem, 
one may show that $R_n^{J_3}=\emptyset$
when $n<n^+$;
the number $x_\Delta^+$ is the unique zero of $W_n(z)$ in $J_3$
in case $n=n^+$;
and $\babs{R_n^{J_3}}=1$ for $n>n^+$.
This completes the proofs of \cref{thm:rr} (i) and \cref{thm:pl} (i).\qed

\medskip

When $c\ge c^+$, recall from \cref{--+-:lem:basic} that $x_\Delta^-<x_A<x_\Delta^+<x_B<x_g^-\le x_g^+$. 
The idea of piecewise interlacing still works for proving \cref{thm:rr} (ii) and \cref{thm:pl} (ii).

\noindent{\bf Proof of \cref{thm:rr} (ii) and \cref{thm:pl} (ii).}
By \cref{--+-:lem:basic},  we have $W_2(x_\Delta^+)<0$. 
By using the facts $W_2(z)=z^2-g(z)$,
it is elementary to verify the truth for $n=1$ that $W_2(z)$ 
has a zero in the interval $J_3$ and another zero in $J_4$. 
Let $n\ge2$. 
By induction, we can suppose that the polynomial $W_{2n-2}(z)$
has zeros $\xi_1$, $\ldots$, $\xi_{2n-2}$ such that
\[
x_\Delta^-
<\xi_1
<\cdots
<\xi_{n-2}
<x_A
<\xi_{n-1}
<\cdots
<\xi_{2n-4}
<x_\Delta^+
<\xi_{2n-3}
<x_g^-
\le x_g^+
<\xi_{2n-2}.
\]
Similar to the proof of \cref{thm:rr} (i), we can deduce that
\[
\sgn\bg{W_{2n}(\xi_j)}
=\begin{cases}
(-1)^{j+1},&\text{if $j\in[n-2]$};\\
(-1)^j,&\text{if $n-1\le j\le 2n-4$};\\
-1,&\text{if $j=2n-3,\,2n-4$}.
\end{cases}
\]
From \cref{--+-:lem:basic}, we see that
\[
W_{2n}(x_\Delta^-)<0,\quad
W_{2n}(x_A)(-1)^n>0,\quad
W_{2n}(x_\Delta^+)<0\rmand
W_{2n}(x_g^\pm)>0.
\]
The remaining of the proof follows from the intermediate value theorem. \qed

\medskip

\Cref{thm:rr} (iii) and \cref{thm:pl} (iii) can be shown along the same vein, and we omit their proofs.

Since $W_3(x_\Delta^-)<0<W_3(x_A)$,
we infer that the polynomial $W_3(z)$ has a zero in the interval $(x_\Delta^-,\,x_A)\subset\R^-$. It is direct to compute that $W_3(0)=bd\ne0$.
\Cref{thm:W3:--+-} explains the extent to which
the condition on the root distribution of~$W_3(z)$ in \cref{thm:rr} (iii) is acceptable.
Recall that $J_g=(x_g^-,\,x_g^+)$.

\begin{thm}\label{thm:W3:--+-}
Let $\{W_n(z)\}_n$ be a \noo\
with $a,b,d<0$ and $c\ge c^+$.
Then the polynomial $W_3(z)$ has a unique negative zero.
If it has a positive zero lying outside the interval $J_g$,
then it has two zeros lying in the interval $(-b/(a+1),\,x_0)$, where
$x_0$ is the positive zero of the quadratic polynomial $A^2(z)+B(z)$, and we have 
\[
x_0<x_\Delta^+,\qquad
a>-1,\qquad
c<\frac{b}{a+1}+\frac{a+1}{b}\cdot d\rmand
c^-<0.
\] 
\end{thm}

\medskip
When the polynomial $W_3(z)$ is real-rooted, 
its root distribution is illustrated in \cref{fig4,fig5}. 
\medskip
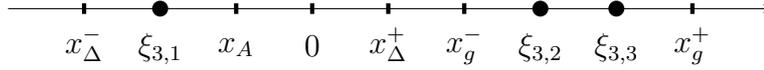
\begin{figure}[htbp]
\begin{tikzpicture}
\draw[->] (0,0) -- (10,0) coordinate (x axis);
\foreach \x in {1,2,...,9} \draw[ultra thick] (\x cm,2pt) -- (\x cm,-2pt);
\fill (2,0) circle (3pt);
\fill (7,0) circle (3pt);
\fill (8,0) circle (3pt);

\draw (1,0) node[below=3pt] {$x_\Delta^-$};
\draw (2,0) node[below=5pt] {$\xi_{3,1}$};
\draw (3,0) node[below=7pt] {$x_A$};
\draw (4,0) node[below=6pt] {$0$};
\draw (5,0) node[below=3pt] {$x_\Delta^+$};
\draw (6,0) node[below=3pt] {$x_g^-$};
\draw (7,0) node[below=5pt] {$\xi_{3,2}$};
\draw (8,0) node[below=5pt] {$\xi_{3,3}$};
\draw (9,0) node[below=3pt] {$x_g^+$};

\end{tikzpicture}
\caption{Illustration for the root distribution of $W_3(z)$ 
when two of its zeros lie in the interval~$J_g$.}\label{fig4}
\end{figure}

\bigskip

\begin{figure}[htbp]
\begin{tikzpicture}
\draw[->] (0,0) -- (12,0) coordinate (x axis);
\foreach \x in {1,2,...,11} \draw[ultra thick] (\x cm,2pt) -- (\x cm,-2pt);
\fill (2,0) circle (3pt);
\fill (6,0) circle (3pt);
\fill (7,0) circle (3pt);

\draw (1,0) node[below=3pt] {$x_\Delta^-$};
\draw (2,0) node[below=5pt] {$\xi_{3,1}$};
\draw (3,0) node[below=7pt] {$x_A$};
\draw (4,0) node[below=6pt] {$0$};
\draw (5,0) node[below=3pt] {$-\frac{b}{a+1}$};
\draw (6,0) node[below=5pt] {$\xi_{3,2}$};
\draw (7,0) node[below=5pt] {$\xi_{3,3}$};
\draw (8,0) node[below=7pt] {$x_0$};
\draw (9,0) node[below=3pt] {$x_\Delta^+$};
\draw (10,0) node[below=3pt] {$x_g^-$};
\draw (11,0) node[below=3pt] {$x_g^+$};

\end{tikzpicture}
\caption{Illustration for the root distribution of $W_3(z)$, 
when $W_3(z)$ has a positive zero lying outside the interval~$J_g$.
It turns out that it has two positive zeros lying in the interval $(-b/(a+1),\,x_0)$.}\label{fig5}
\end{figure}
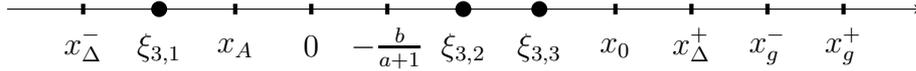

\begin{proof}
By \cref{rec11}, 
one may compute that $W_3(z)=a^2z^3+Uz^2+Vz+bd$, where
$U=2ab+ac+c$ and $V=ad+bc+b^2+d$.
Suppose that $c\ge c^+$. 
Let $\{x_1,x_2,x_3\}$ be the zero set of $W_3(z)$ such that $x_1\in(x_\Delta^-,\,x_A)$.
By Vi\`eta's formula, we have 
\begin{align}
x_1+x_2+x_3&=-U/a^2,\label{e1:W3}\\
x_1x_2+x_1x_3+x_2x_3&=V/a^2,\label{e2:W3}\\
x_1x_2x_3&=-bd/a^2>0.\label{e3:W3}
\end{align}
From \cref{e3:W3} and the fact $x_1<0$,
we infer that the zeros $x_2$ and $x_3$ have the same sign. 
We shall show that both $x_2$ and $x_3$ are positive.

Assume, to the contrary, that $x_2,x_3<0$.
Then \cref{e1:W3} implies $U>0$, and \cref{e2:W3} implies $V>0$. 
On one hand, since
\[
ad+b^2+d=V-bc>-bc\ge-bc^+=-b(2\sqrt{d(a-1)}-b),
\]
that is, 
\begin{equation}\label[ineq]{pf11}
(a+1)\sqrt{-d}<2b\sqrt{1-a},
\end{equation}
we infer that $a+1<0$. 
On the other hand,
\[
2ab=U-(a+1)c>-(a+1)c\ge-(a+1)c^+=-(a+1)(2\sqrt{d(a-1)}-b),
\]
that is, $b\sqrt{1-a}<2(a+1)\sqrt{-d}$.
Combining it with \cref{pf11}, we find $a+1>0$, a contradiction.
Therefore, both the zeros $x_2$ and $x_3$ are positive.
This proves the uniqueness of the negative zero $x_1$.

Suppose tha $\{x_2,\,x_3\}\backslash J_g\ne\emptyset$.
From \cref{--+-:lem:basic}, we see that $W_n(x_g^\pm)>0$.
By the intermediate value theorem, the number of zeros of $W_3(z)$ in $J_g$ is even.
Thus $\{x_2,\,x_3\}\cap J_g=\emptyset$.
Let $\xi\in\{x_2,x_3\}$. Then $\xi>0$.
Since the polynomial $g(z)$ is quadratic with leading coefficient $1-a>0$,
and since $\xi\not\in J_g$, we infer that $g(\xi)>0$.

Let $F(z)=A^2(z)+B(z)$. By \cref{rec11}, one may deduce that
\begin{align}
0&=W_3(\xi)=A(\xi)W_2(\xi)+B(\xi)\cdot\xi
=A(\xi)\cdot \bg{A(\xi)\cdot\xi+B(\xi)}+B(\xi)\cdot\xi\notag\\
&=A^2(\xi)\cdot \xi+B(\xi)\cdot\bg{A(\xi)+\xi}\label{x:W3}\\
&=F(\xi)\cdot\xi+A(\xi)B(\xi).\label{pf:F}
\end{align}
It follows that $F(\xi)\ne0$ and $\xi=-A(\xi)B(\xi)/F(\xi)$.
Thus $g(\xi)=-B^3(\xi)/F^2(\xi)$.
Since $g(\xi)>0$, we deduce that $B(\xi)<0$. 
By \cref{x:W3}, we infer that $A(\xi)+\xi>0$, which implies that $a>-1$ and $\xi>-b/(a+1)$.

Since $x_A<0<\xi$, we have $A(\xi)<0$. By \cref{pf:F}, we infer that $F(\xi)<0$.
Note that $F(z)=a^2z^2+(2ab+c)z+(b^2+d)$ is quadratic with positive leading coefficient.
The fact $F(\xi)<0$ implies that $F(z)$ is real-rooted and it has a zero larger than $\xi>0$.  
Applying Vi\'eta's formula on $F(z)$, we obtain that $F(z)$ has a negative zero.
Thus $F(z)$ has a unique positive zero, say, $x_0$, and $x_0>\xi$.
From \cref{--+-:lem:basic}, we see that $x_\Delta^+<x_B$. Thus
\[
F(x_\Delta^+)=\Delta(x_\Delta^+)-3B(x_\Delta^+)=-3B(x_\Delta^+)>0.
\]
Thus $x_0<x_\Delta^+$. 
Since the polynomial $F(z)$ is increasing in the interval $(0,+\infty)$ and $F(\xi)<0$, 
we infer that $F(-b/(a+1))<0$, that is,
\[
c<\frac{b}{a+1}+\frac{a+1}{b}\cdot d,\qquad\text{or equivalently,}\qquad d<\frac{bc(a+1)-b^2}{(a+1)^2}.
\]
It follows that $d<-b^2$, and $c^-<0$ from definition.
This completes the proof.
\end{proof}

Continuing \cref{eg2}, we illustrate \cref{thm:W3:--+-} by presenting
that $W_3(z)$ may have two positive zeros 
in the interval $(-b/(a+1),\,x_0)$ and that $W_3(z)$ may have non-real zeros.

\begin{eg}\label{eg3}
Let $c\in\R$. Let $W_n(z)$ be polynomials defined by the recurrence
\[
W_n(z)=(-0.3z-1)W_{n-1}(z)+(cz-60)W_{n-2}(z),
\]
with $W_0(z)=1$ and $W_1(z)=z$.
Then $c^-\approx -16.6$ and $c^+\approx 18.6$. 
\begin{itemize}
\itemsep 4pt
\item
If $c=20$, then $c>c^+$, and every polynomial $W_{2n}(z)$ is real-rooted by \cref{thm:rr} (ii). For example, $W_4(z)$ has the real zeros 
\[
\xi_{4,1}\approx-423.39, \qquad
\xi_{4,2}\approx2.89, \qquad
\xi_{4,3}\approx3.67, \qquad
\xi_{4,4}\approx29.03.
\]
Moreover, we can compute that
\[
-b/(a+1)\approx 1.42,\qquad
\xi_{3,2}\approx 1.61,\qquad 
\xi_{3,3}\approx 2.48,\qquad
x_0\approx2.82.
\]
Thus $-b/(a+1)<\xi_{3,2}<\xi_{3,3}<x_0$.
In this case, the polynomial $W_5(z)$ has real zeros approximately
$2.93$, $-47.44$, and $-574.73$,
and non-real zeros approximately $2.95\pm1.52\,\mathrm{i}$. 
\item
If $c=30$, then $W_3(z)$ has non-real zeros approximately $1.63\pm0.30\,\mathrm{i}$.
\end{itemize}

\end{eg}

\section{Concluding remarks}\label{sec:conj}

With a bit more work on the root distribution of the polynomial $W_3(z)$ in \cref{thm:W3:--+-},
one may obtain \cref{thm:c:large}. 

\begin{thm}\label{thm:c:large}
Let $\{W_n(z)\}_n$ be a \noo\ with $a,b,d<0$ and $c\in\R$.
Then every polynomial $W_n(z)$ is real-rooted if the number $c$ is sufficiently large.
\end{thm}
\begin{proof}
By routine computation, one sees that the discriminant of the polynomial~$W_3(z)$ 
is a quartic polynomial in $c$ with a positive leading coefficient. 
Thus there exists $N$ such that the discriminant is positive as if $c>N$. 
Let 
\[
c^*=\max\bgg{c^+,\ N,\ \frac{b}{a+1}+\frac{a+1}{b}\cdot d}. 
\]
Suppose that $c>c^*$. Then $W_3(z)$ is real-rooted, because of the positivity of its discriminant. 
By \cref{thm:W3:--+-}, $W_3(z)$ has two positive zeros in the interval $J_g$. 
By \cref{thm:rr} (ii) and (iii), every polynomial $W_n(z)$ is real-rooted. 
\end{proof}

We end this paper by proposing \cref{conj:only1complex}, 
which is partially supported by \cref{thm:c:large}.

\begin{conj}\label{conj:only1complex}
Let $\{W_n(z)\}_n$ be a \noo\ with $a,b,d<0$ and $c>0$. 
Then every polynomial $W_n(z)$ has at most two non-real zeros. 
\end{conj}

We notice that the polynomial sequence $\{W_n(z)\}_n$ satisfying \cref{rec11} 
can be written as $W_n(z)=f_n(z)+g_n(z)$,
where $\{f_n(z)\}_n$ is a polynomial sequence whose root distribution has been studied by Tran~\cite{Tra14} 
and $\{g_n(z)\}_n$ is a polynomials sequence whose root distribution has been studied by Mai~\cite{Mai18}.
We did not find a modification of any proof of theirs 
which may show \cref{thm:rr}.


\begin{thebibliography}{99}
\itemsep 2pt
\bibitem{ARR99B}
G.E. Andrews, A. Richard, and R. Ranjan,
Special Functions, Camb. Univ. Press, Cambridge, 1999.

\bibitem{BB09}
J. Borcea and P. Br\"anden, 
The Lee-Yang and P\'olya-Schur Programs. II. Theory of Stable Polynomials and Applications, 
Comm. Pure Appl. Math. 62(12) (2009), 1595--1631.

\bibitem{BDR04}
C. Brezinski, K.A. Driver, M. Redivo-Zaglia,
Quasi-orthogonality with applications 
to some families of classical orthogonal polynomials,
\ANM 48 (2004), 157--168.

\bibitem{BKW75}
S. Beraha, J. Kahane, and N.J. Weiss,
Limits of zeroes of recursively defined polynomials,
Proc. Natl. Acad. Sci. 72(11) (1975), 4209.

\bibitem{BKW78}
---,
Limits of zeroes of recursively defined families of polynomials,
Adv. in Math. Suppl. Stud.~1 (1978), 213--232.

\bibitem{BM06}
P. Bleher, R. Mallison Jr., 
Zeros of sections of exponential sums,
Int. Math. Res. Not. 2006 (2006), 1--49, Article ID 38937.

\bibitem{BG07}
R. Boyer and W.M.Y. Goh,
On the zero attractor of the Euler polynomials,
Adv. in Appl. Math. 38(1) (2007), 97--132.

\bibitem{BG08}
---,
Appell polynomials and their zero attractors,
Gems in experimental mathematics, 69--96, 
Contemp. Math. 517 (2008), 69--96.
Amer. Math. Soc., Providence, RI, 2010.

\bibitem{GMTW16-01}
J.L. Gross, T. Mansour, T.W. Tucker, and D.G.L. Wang,
Root geometry of polynomial sequences I: Type $(0,1)$,
\JMAA 433(2) (2016), 1261--1289.

\bibitem{GMTW16-10}
---,
Root geometry of polynomial sequences II: type $(1,0)$, 
\JMAA 441(2) (2016), 499--528.

\bibitem{GHR09}
W. Goh, M.X. He, P.E. Ricci,
On the universal zero attractor of the Tribonacci-related polynomials,
Calcolo 46 (2009), 95--129. 

\bibitem{HS94}
M.X. He and E.B. Saff, 
The Zeros of Faber Polynomials for an m-Cusped Hypocycloid, 
J. Approx. Theory, 78(3) (1994), 410--432. 

\bibitem{JW17X}
D.D.D. Jin and D.G.L. Wang,
Common zeros of polynomials satisfying a recurrence of order two,
{\tt arXiv[math.CO]:1712.04231}.

\bibitem{JT09}
K. Jordaan and F. To\'okos,
Interlacing theorems for the zeros of some orthogonal polynomials from different sequences, 
\ANM 59 (2009), 2015--2022.

\bibitem{KL93B}
V.L. Kocic and G. Ladas,
Global Behavior of Nonlinear Difference Equations of Higher Order with Applications,
Mathematics and Its Applications, vol. 256,
Kluwer Academic Publishers, Dordrecht/Boston/London, Springer, 1993.

\bibitem{LW07} 
L.L. Liu and Y. Wang,
A unified approach to polynomial sequences with only real zeros,
\AAM 38 (2007), 542--560.

\bibitem{Mai18}
A.D. Mai,
Exceptional zeros of polynomials satisfying a three-term recurrence,
J. Pure Appl. Algebra 222(3) (2018), 534--545. 

\bibitem{RS02B}
Q.I. Rahman, G. Schmeisser, 
Analytic Theory of Polynomials,  
London Math. Soc. Monogr. New Ser. 26.,
The Clarendon Press, Oxford Univ. Press, Oxford, 2002.

\bibitem{ST92B}
H. Stahl and V. Totik,
General Orthogonal Polynomials, Camb. Univ. Press, Cambridge, 1992.

\bibitem{StaW}
R.P. Stanley, 
{\tt http://www-math.mit.edu/\char126 rstan/zeros}. 

\bibitem{Sta00}
R.P. Stanley, 
Positivity problems and conjectures in algebraic combinatorics,
in V. Arnold, M. Atiyah, P. Lax, and B. Mazur (Eds.), Mathematics: frontiers and perspectives, Providence: Amer. Math. Soc., 2000, pp. 295--319.

\bibitem{Sok04}
A. Sokal,
Chromatic roots are dense in the whole complex plane, 
Combin. Probab. Comput. 13(2) (2004), 221--261.

\bibitem{Sze24}
G. Szeg\H o,
\"Uber eine Eigenschaft der Exponentialreihe, 
Sitzungsber. Berl. Math. Ges. 23 (1924), 50--64.

\bibitem{Sze59B}
G. Szeg\H o,
Orthogonal Polynomials, \AMS, New York, 1959.

\bibitem{Tra14}
K. Tran, 
Connections between discriminants and the root distribution of polynomials 
with rational generating function, \JMAA 410 (2014), 330--340.



\end{thebibliography}
\end{document}